\documentclass[11pt]{article}
\usepackage{amsthm,amsfonts, amsbsy, amssymb,amsmath,graphicx}
\usepackage{graphics}
\usepackage[english]{babel}
\usepackage{graphicx}
\usepackage{lineno}
\usepackage{enumerate}
\usepackage{tikz}
\usepackage{tkz-graph}
\usepackage{tkz-berge}
\usepackage{hyperref}
\usepackage{color}
\usepackage{tikz}

\hypersetup{colorlinks=true}

\hypersetup{colorlinks=true, linkcolor=blue, citecolor=blue,urlcolor=blue}

\input epsf
\usepackage{epic,latexsym,amssymb}
\usepackage{tikz}

\usepackage{tikz,epsf}

\usepackage{amsfonts}
\usepackage{amscd}
\usepackage{amsmath}
\usepackage{graphicx}
\usepackage{color}
\usepackage{caption,subcaption}

\newtheorem{theorem}{Theorem}
\newtheorem{observation}[theorem]{Observation}
\newtheorem{proposition}[theorem]{Proposition}

\newtheorem{corollary}[theorem]{Corollary}

\newtheorem{remark}[theorem]{Remark}

\tikzstyle{vertex}=[circle, draw, inner sep=0pt, minimum size=6pt]

\newcommand{\QEDmark}{\mbox{\textsc{qed}}}
\newcommand{\proofStarter}[1]{\textsc{#1}}

\begin{document}

\title{Covering Italian domination in graphs}
\date{}
\author {
Abdollah Khodkar$^a$, Doost Ali Mojdeh$^b$, Babak Samadi$^b$\thanks{Corresponding author} and Ismael G. Yero$^c$\vspace{1.5mm}\\
$^a$Department of Mathematics\\
University of West Georgia, Carrollton, GA 30118, USA\\
{\tt akhodkar@westga.edu}\vspace{1.5mm}\\
$^b$Department of Mathematics\\
University of Mazandaran, Babolsar, Iran\\
{\tt damojdeh@umz.ac.ir$^*$}\\
{\tt samadibabak62@gmail.com}\vspace{1.5mm}\\
$^c$Departamento de Matem\'{a}ticas\\
Universidad de C\'{a}diz, Algeciras, Spain\\
{\tt ismael.gonzalez@uca.es}\vspace{3mm}\\
}
\date{}
\maketitle

\begin{abstract}
For a graph $G=(V(G),E(G))$, an Italian dominating function (ID function) of $G$ is a function $f:V(G)\rightarrow \{0,1,2\}$ such that for each vertex $v\in V(G)$ with $f(v)=0$, $f(N(v))\geq2$, that is, either there is a vertex $u \in N(v)$ with $f (u) = 2$ or there are two vertices $x,y\in N(v)$ with $f(x)=f(y)=1$. A function $f:V(G)\rightarrow \{0,1,2\}$ is a covering Italian dominating function (CID function) of $G$ if $f$ is an ID function and $\{v\in V(G)\mid f(v)\neq0\}$ is a vertex cover set. The covering Italian domination number (CID number) $\gamma_{cI}(G)$ is the minimum weight taken over all CID functions of $G$.

In this paper, we study the CID number in graphs. We show that the problem of computing this parameter is NP-hard even when restricted to some well-known families of graphs, and find some bounds on this parameter. We characterize the family of graphs for which their CID numbers attain the upper bound twice their vertex cover number as well as all claw-free graphs whose CID numbers attain the lower bound half of their orders. We also give the characterizations of some families of graphs with small or large CID numbers.
\end{abstract}




\section{Introduction}

Throughout this paper, we consider $G$ as a finite simple graph with vertex set $V(G)$ and edge set $E(G)$. We use \cite{we} as a reference for terminology and notation which are not explicitly defined here. The {\em open neighborhood} of a vertex $v$ is denoted by $N(v)$, and its {\em closed neighborhood} is $N[v]=N(v)\cup \{v\}$. The {\em minimum} and {\em maximum degrees} of $G$ are denoted by $\delta(G)$ and $\Delta(G)$, respectively.  Given the subsets $A,B\subseteq V(G)$, by $[A,B]$ we mean the set of all edges with one end point in $A$ and the other in $B$. Also, $deg_{H}(v)=|N(v)\cap V(H)|$ by taking $H$ as a subgraph of $G$. With the same assumption, by $G-H$ we mean the graph obtained from $G$ by removing all the vertices in $V(H)$, and edges incident with such vertices. Finally, for a given set $S\subseteq V(G)$, by $G[S]$ we represent the subgraph induced by $S$ in $G$.

A subset $S\subseteq V(G)$ is said to be {\em independent} if there is no edge with both end points in $S$. The {\em independence number} $\alpha(G)$ is the maximum cardinality among all independent sets of $G$. A {\em vertex cover} of $G$ is a set $Q\subseteq V(G)$ that contains at least one endpoint of every edge. The {\em vertex cover number} $\beta(G)$ is the minimum cardinality among all vertex cover sets of $G$. For any parameter $p$ of $G$, by a $p(G)$-set we mean a set of cardinality $p(G)$.

For a function $f:V(G)\rightarrow\{0,1,2\}$, we let $V^{f}_{i}=\{v\in V(G)\mid f(v)=i\}$ for each $i=0,1,2$ (we simply write $V_{i}$ if there is no ambiguity with respect to the function $f$). We call $\omega(f)=f(V(G))=\sum_{v\in V(G)}f(v)$ as the {\em weight} of $f$. A \textit{Roman dominating function} (RD function) of a graph $G$ is a function $f:V\rightarrow \{0,1,2\}$ such that if $v\in V_0$ for some $v\in V(G)$, then there exists $w\in N(v)$ such that $w\in V_2$. The minimum weight taken over all RD functions for $G$ is called the \textit{Roman domination number} of $G$, denoted by $\gamma_{R}(G)$. This concept was formally defined by Cockayne et al. in \cite{cdhh}.

Chellali et al \cite{chhm} introduced an Italian dominating function (also known as Roman $\{2\}$-dominating function) $f$ as follows. An \textit{Italian dominating function} (ID function) is a function $f:V(G)\rightarrow \{0,1,2\}$ with the property that for every vertex $v\in V(G)$ with $f(v)=0$, it follows $f(N(v))\geq 2$. That is, either there is a vertex $u\in N(v)$ with $f(u)=2$ or at least two vertices $x,y\in N(v)$ with $f(x)=f(y)=1$. The \textit{Italian domination number} (ID number) $\gamma_{I}(G)$ is the minimum weight among all ID functions of $G$.

The concept of Roman domination was motivated by the article of Ian Stewart entitled ``Defend the Roman Empire!" (\cite{s}), published in {\it Scientific American}. The idea was that the values $1$ and $2$ represent the number of Roman legions stationed at a location $v$. A location $u\in N(v)$ is considered to be {\em unsecured} if no legion is stationed there ($f(u)=0$). The unsecured location $u$ can be secured by sending a legion to $u$ from an adjacent location $v$. But a legion cannot be sent from a location $v$ if doing so leaves that location unsecured (if $f(v)=1$). Thus, two legions must be stationed at a location ($f(v)=2$) before one of the legions can be sent to a neighboring location. In terms of the Roman Empire, the Italian dominating strategy requires that every location with no legion has a neighboring location with two legions, or at least two neighboring locations with one legion each.

The existence of two adjacent locations with no guards can jeopardize them. In fact, they would be considered more vulnerable. One improved situation for a location with no guards is to be surrounded by locations at which guards are stationed. This motivates us to consider an ID function $f$ for which the vertices assigned $0$ under $f$ are independent or, equivalently, the set of vertices assigned $1$ or $2$ is a vertex cover set in the graph modeling. This provides a more flexible and stronger level of defense. More formally, we have the following definition. A function $f:V(G)\rightarrow \{0,1,2\}$ is a \textit{covering Italian dominating function} (CID function) of $G$ if $f$ is an ID function and $V_{1}\cup V_{2}$ is a vertex cover set (or, equivalently, $V_0$ is an independent set). The \textit{covering Italian domination number} (CID number) $\gamma_{cI}(G)$ is the minimum weight taken over all CID functions of the graph $G$. We must remark that this concept was already introduced in \cite{fan} under the name of \textit{outer independent Italian domination}. However, we prefer to use the terminology of covering since it looks more natural for us and our purposes in this work, and one of the first reasons for this is the next observation which was precisely first given in  \cite{fan}.

\begin{observation}\emph{(\cite{fan})}\label{obser}
For any graph $G$ with minimum degree at least two, $\gamma_{cI}(G)=\beta(G)$.
\end{observation}

From the observation above one could think that probably the graphs $G$ with minimum degree at least two are the only ones for which $\gamma_{cI}(G)=\beta(G)$. However, this is not true, as we can see from the example given in Figure \ref{cover-cid}. In such a case, the bold vertices form a vertex cover set of the minimum cardinality, and the labels given to each vertex form a CID function of the minimum weight. Thus, the equality $\gamma_{cI}(G)=\beta(G)$ is satisfied for such a graph with minimum degree one.

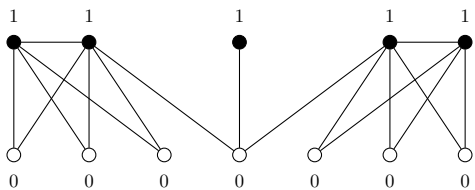
\begin{figure}[h]
\centering
\begin{tikzpicture}[scale=.5, transform shape]
\node [draw, shape=circle] (x1) at  (0,0) {};
\node [draw, shape=circle] (y1) at  (2,0) {};
\node [draw, shape=circle] (z1) at  (4,0) {};

\node [draw, shape=circle] (x2) at  (8,0) {};
\node [draw, shape=circle] (y2) at  (10,0) {};
\node [draw, shape=circle] (z2) at  (12,0) {};

\node [draw, shape=circle,fill=black] (c1) at  (2,3) {};
\node [draw, shape=circle,fill=black] (c2) at  (10,3) {};

\node [draw, shape=circle] (w1) at  (6,0) {};
\node [draw, shape=circle,fill=black] (w2) at  (6,3) {};

\node [draw, shape=circle,fill=black] (x11) at  (0,3) {};
\node [draw, shape=circle,fill=black] (z22) at  (12,3) {};

\node [scale=1.3] at (0,-0.7) {0};
\node [scale=1.3] at (2,-0.7) {0};
\node [scale=1.3] at (4,-0.7) {0};
\node [scale=1.3] at (6,-0.7) {0};
\node [scale=1.3] at (8,-0.7) {0};
\node [scale=1.3] at (10,-0.7) {0};
\node [scale=1.3] at (12,-0.7) {0};

\node [scale=1.3] at (0,3.7) {1};
\node [scale=1.3] at (2,3.7) {1};
\node [scale=1.3] at (6,3.7) {1};
\node [scale=1.3] at (10,3.7) {1};
\node [scale=1.3] at (12,3.7) {1};

\draw(x1)--(x11)--(y1);
\draw(x11)--(z1);
\draw(x2)--(z22)--(y2);
\draw(z22)--(z2);
\draw(x1)--(c1)--(y1);
\draw(x11)--(c1)--(z1);
\draw(c1)--(w1)--(c2)--(z22);
\draw(w2)--(w1);

\draw(y2)--(c2)--(x2);
\draw(c2)--(z2);

\end{tikzpicture}
\caption{A graph $G$ with minimum degree one and $\gamma_{cI}(G)=\beta(G)$.}\label{cover-cid}
\end{figure}

One must notice that despite the fact that this graph $G$ of Figure \ref{cover-cid} with minimum degree one satisfies $\gamma_{cI}(G)=\beta(G)$, does not mean that every vertex cover set of such a graph $G$ ``provides'' a CID function. If we exchange the bold vertex of degree one in the vertex cover set by its neighbor, then we get again a vertex cover set, but a similar labeling as the one shown in the figure does not produce a CID function. Is there then any connection between this fact and the reaching or not of the equality $\gamma_{cI}(G)=\beta(G)$?

This paper is organized as follows: We investigate the covering Italian domination in graphs. We show that the problem of computing the CID number is NP-hard even for some well-known family of graphs and give some comment on the approximation of this problem. In Section $3$, we give a characterization of all graphs $G$ for which $\gamma_{cI}(G)=2\beta(G)$ which is the limit case of a bound given in Section $2$. A sharp lower bound on the CID number of the $K_{1,r}$-free graphs is given in Section $4$. Moreover, we characterize the claw-free graphs attaining the lower bound. Finally, we give the characterizations of graphs $G$ with small or large values for $\gamma_{cI}(G)$.


\section{Complexity results}

We consider the problem of deciding whether a graph $G$ has its CID number at most a given integer. That is stated in the following decision problem.

$$\begin{tabular}{|l|}
  \hline
  \mbox{CID problem}\\
  \mbox{INSTANCE: A graph $G$ and an integer $k\leq|V(G)|$.}\\
  \mbox{QUESTION: Is $\gamma_{cI}(G)\leq k$?}\\
  \hline
\end{tabular}$$

Our aim is to show that the problem is NP-complete for cubic graphs and planar graphs with maximum degree at most four as well as triangle-free graphs. To this end, we make use of the well-known VERTEX COVER PROBLEM (VC problem) which is known to be NP-complete from \cite{gj}, and from \cite{garey} for the specific case of cubic graphs.

$$\begin{tabular}{|l|}
  \hline
  \mbox{VC problem}\\
  \mbox{INSTANCE: A graph $G$ and an integer $j\leq|V(G)|$.}\\
  \mbox{QUESTION: Is $\beta(G)\leq j$?}\\
  \hline
\end{tabular}$$

Moreover, the problem above remains NP-complete even when restricted to cubic graphs (see \cite{garey}), some planar graphs and triangle-free graphs (see \cite{gj}). Indeed, we have the following result.

\begin{theorem}\emph{(\cite{gj,garey})}
The VC problem is NP-complete even when restricted to cubic graphs, planar graphs of maximum degree at most three and triangle-free graphs.
\end{theorem}

\begin{theorem}\label{planar}
The CID problem is NP-complete even when restricted to cubic graphs, planar graphs with maximum degree at most four and triangle-free graphs.
\end{theorem}

\begin{proof}
The problem clearly belongs to NP since checking that a given function is indeed a CID function with weight at most $k$ can be done in polynomial time. We first consider the case when $G$ is a cubic graph. In this situation, by Observation \ref{obser}, we know that $\gamma_{cI}(G)=\beta(G)$. Thus, a reduction from the VC problem to the CID problem can be easily noted. This means that the CID problem is NP-complete for cubic graphs.

On the other hand, let $G$ be a planar graph with $V(G)=\{v_{1},\dots,v_{n}\}$ and maximum degree $\Delta(G)\leq 3$ (a triangle-free graph). Let $G'=G\odot K_1$ (the corona of $G$), which is obtained from $G$ by adding $n$ new vertices $u_{1},\dots,u_{n}$ and joining $u_{i}$ to $v_{i}$ for all $1\leq i\leq n$. Clearly, $G'$ is a planar graph with $|V(G')|=2n$ and $\Delta(G')\leq 4$ (a triangle-free graph as well).

Let $f$ be a $\gamma_{cI}(G')$-function. Clearly, $1\leq f(v_{i})+f(u_{i})\leq2$ for each $1\leq i\leq n$. If $f(v_{i})+f(u_{i})=2$ for some $1\leq i\leq n$, then we may without loss of generality assume that $f(v_{i})=2$ and $f(u_{i})=0$. Moreover, $f(u_{i})=1$ whenever $f(v_{i})+f(u_{i})=1$. Let $X=\{1\leq i\leq n\mid f(v_{i})+f(u_{i})=1\}$. Hence,
\begin{equation}\label{Tak}
\gamma_{cI}(G')=\sum_{i\notin X}(f(v_{i})+f(u_{i}))+\sum_{i\in X}(f(v_{i})+f(u_{i}))=2n-|X|.
\end{equation}

On the other hand, $|X|\leq \alpha(G)$ as the vertices $v_{i}$, for which $i\in X$, are assigned $0$ under $f$. So, $\gamma_{cI}(G')\geq 2n-\alpha(G)=n+\beta(G)$ by (\ref{Tak}) and the well-known Gallai theorem (\cite{we}) which states that $\alpha(H)+\beta(H)=n$, for any graph $H$ of order $n$.

Let $I$ be an $\alpha(G)$-set. We can observe that the assignment $0$ to the vertices in $I$ and $1$ to the other vertices of $G'$ defines a CID function of $G'$ with weight $n+\beta(G)$. Therefore, $\gamma_{cI}(G')\leq n+\beta(G)$ and so, $\gamma_{cI}(G')= n+\beta(G)$.

Let $j=n+k$. Then $\gamma_{cI}(G')\leq j$ if and only if $\beta(G)\leq k$, which completes the reduction. Since the VC problem is NP-complete for planar graphs of maximum degree at most three and triangle-free graphs, we deduce that the CID problem is also NP-complete for planar graphs of maximum degree at most four as well as triangle-free graphs.
\end{proof}

As a consequence of Theorem \ref{planar}, we conclude that the problem of computing the CID number is NP-hard, even when restricted to cubic graphs, planar graphs with maximum degree at most four and triangle-free graphs. In consequence, it would be desirable to bound the CID number in terms of several different invariants of the graph.

From Observation \ref{obser}, and by the fact that computing the vertex cover number of bipartite graphs can be polynomially done (by using for instance the well-known K\"{o}nig's Theorem - see also \cite{gj}), we can deduce that computing the CID number of bipartite graphs of minimum degree at least two can be polynomially done also.

In addition, we can also derive some approximation comments concerning computing the CID number of graphs. To do so, we need the following remark.

\begin{remark}\label{vertex-cover}
For any connected graph $G$ of order at least two, $\beta(G)\le \gamma_{cI}(G)\leq 2\beta(G)$.
\end{remark}

\begin{proof}
Let $S$ be an independent set in $G$ of cardinality $\alpha(G)$. Let $f:V(G)\rightarrow \{0,1,2\}$  be a function with $f(v)=2$ for $v\in V(G)\setminus S$, and $f(x)=0$ for the other vertices $x$. It is easy to see that $f$ is a CID function of $G$ with weight $2(n-\alpha(G))$. Since $\alpha(G)+\beta(G)=n$, the upper bound follows. The lower bound clearly follows since the set of vertices labeled with a positive number is a vertex cover of $G$.
\end{proof}

It is well-known that (see for instance \cite{erdos}) the vertex cover problem can be approximated within a factor of two by using a  simple greedy algorithm. This, together with Remark \ref{vertex-cover}, allow to claim that computing the CID number of graphs can also be approximated within a constant factor.


\section{Graphs $G$ with $\gamma_{cI}(G)=2\beta(G)$}

In what follows, we dedicate some particular attention to the upper bound given in Remark \ref{vertex-cover}. To this end, we need the following. Let $G_{1},\dots,G_{t}$ be the components of a graph $G$. Since $\gamma_{cI}(G)=\sum_{i=1}^{t}\gamma_{cI}(G_{i})$, it follows that $\gamma_{cI}(G)=2\beta(G)$ if and only if $\gamma_{cI}(G_{i})=2\beta(G_{i})$ for each $1\leq i\leq t$. So, in order to characterize all graphs for which their CID numbers equal twice its vertex cover number, it suffices to do it when $G$ is connected, and thus, our results stand by this assumption.

We introduce an infinite family of graphs as follows so as to characterize the graphs attaining the mentioned bound. Let $\mathcal{G}$ be a family of graphs $G$ constructed from a graph $H$ by joining every vertex of $H$ to at least two vertices in $G-H$ such that at least one of them is a leaf. Moreover, we let any $k$ independent vertices in $H$, among those vertices having only one leaf in $G-H$, have at least $k$ non-leaf independent neighbors in $G-H$ (see Figure \ref{fig11}). Note that $V(G)\setminus V(H)$ is an independent set.\vspace{1mm}

\begin{figure}[h]\vspace{-18mm}
\tikzstyle{every node}=[circle, draw, fill=white!, inner sep=0pt,minimum width=.16cm]
\begin{center}
\begin{tikzpicture}[thick,scale=.6]
  \draw(0,0) { 

+(-10.5,-5) node{}
+(-8,-5) node{}

+(-5.5,-4.5) node{}
+(-5.5,-3) node{}
+(-5.5,-5.75) --+(-5.5,-4.5) --+(-5.5,-3)
+(-5.5,-4.5) --+(-1,-6.75)

+(-3,-5) node{}
+(1,-5) node{}
+(-1,-4) node{}

+(-11.2,-3.5) node{}
+(-10.5,-3.5) node{}
+(-9.8,-3.5) node{}

+(-8,-3.5) node{}

+(-3,-3.5) node{}
+(1,-3.5) node{}
+(-1.5,-3) node{}
+(-0.5,-3) node{}

+(-5.5,-5.75) node{}
+(-5.5,-6.75) node{}

+(-1,-6.75) node{}

+(-10.5,-5) -- +(-8,-5) --+(-8,-3.5)
+(-3,-5) -- +(1,-5) -- +(-1,-4) -- +(-3,-5)
+(-10.5,-5) --+(-11.2,-3.5)
+(-10.5,-5) --+(-10.5,-3.5)
+(-10.5,-5) --+(-9.8,-3.5)
+(-3,-5) --+(-3,-3.5)
+(1,-5) --+(1,-3.5)
+(-0.5,-3) --+(-1,-4) --+(-1.5,-3)
+(-8,-5) --+(-5.5,-5.75) --+(-3,-5)
+(-8,-5) --+(-5.5,-6.75) --+(1,-5)
+(-3,-5) --+(-1,-6.75) --+(1,-5)

+(-10.5,-5.5) node[rectangle, draw=white!0, fill=white!100]{$v_{1}$}
+(-8.1,-5.5) node[rectangle, draw=white!0, fill=white!100]{$v_{2}$}
+(-6,-4.5) node[rectangle, draw=white!0, fill=white!100]{$v_{3}$}
+(-3.4,-4.6) node[rectangle, draw=white!0, fill=white!100]{$v_{4}$}
+(1.1,-5.5) node[rectangle, draw=white!0, fill=white!100]{$v_{5}$}
+(-0.4,-3.9) node[rectangle, draw=white!0, fill=white!100]{$v_{6}$}

};
\end{tikzpicture}
\end{center}\vspace{-5mm}
\caption{A representative member of $\mathcal{G}$ in which $H$ is the subgraph induced by $\{v_{1},\cdots,v_{6}\}$.}\label{fig11}
\end{figure}
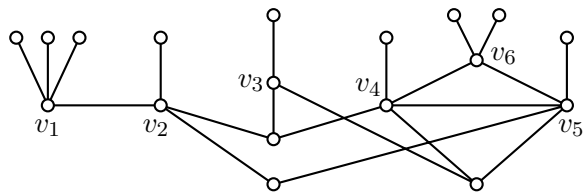

We might remark that the above construction implies that the set $V(G)\setminus V(H)$ produces an independent set in $G$ of cardinality $\alpha(G)$. From now on we are then in a position to present the main theorem of this section.

\begin{theorem}\label{pro8}
For a connected graph $G$ of order at least three, $\gamma_{cI}(G)=2\beta(G)$ holds if and only if $G\in \mathcal{G}$.
\end{theorem}

\begin{proof}
Let $G\in \mathcal{G}$. Note that, by construction, $S=V(G)\setminus V(H)$ is an $\alpha(G)$-set. Therefore, $V(H)$ is a $\beta(G)$-set. If there are at least two leaves at every vertex of $H$, then $f(v)=2$ for $v\in V(H)$, and $f(x)=0$ for the other vertices $x$ defines a $\gamma_{cI}(G)$-function with weight $2\beta(G)$.

Now suppose that some vertices $v\in V(H)$ have only one leaf in $S$.
Let $\gamma_{cI}(G)<2\beta(G)$ and let $g$ be a $\gamma_{cI}(G)$-function with the property that $|V_{2}^{g}\cap V(H)|$ is maximized. Since $\omega(g)<2\beta(G)$, it follows that there exists a vertex $v\in V(H)$ with wight $0$ or $1$ under $g$. Let $g(v)=1$ and let $u\in V(G)\setminus V(H)$ be a leaf adjacent to $v$. This implies that $g(u)=1$, necessarily. Now $(g'(u),g'(v))=(0,2)$ and $g'(x)=g(x)$ for the other vertices defines a $\gamma_{cI}(G)$-function for which $|V_{2}^{g'}\cap V(H)|>|V_{2}^{g}\cap V(H)|$ which contradicts our choice of $g$. Therefore, $g(v)=0$. Since $V_{0}^{g}$ is independent, $g(u)\geq1$. If $g(u)=2$, then the above-mentioned function $g'$ is again a $\gamma_{cI}(G)$-function, a contradiction. Therefore, $g(u)=1$. In conclusion, the above argument implies that:\vspace{1.5mm}\\
$\bullet$ all vertices in $V(H)$ are assigned $0$ or $2$ under $g$,\vspace{1.5mm}\\
$\bullet$ for any vertex $x\in V(H)$ having more than one leaf in $V(G)\setminus V(H)$, we have $g(x)=2$ and $g(y)=0$ for all leaves $y$ adjacent to $x$, and\vspace{1.5mm}\\
$\bullet$ for any vertex $x\in V(H)$ with precisely one leaf $y$ in $V(G)\setminus V(H)$, we have $g(x)=0$ (resp. $g(x)=2$), and $g(y)=1$ (resp. $g(y)=0$).\vspace{1.5mm}

Since $deg_{G-H}(v)\geq2$, there exists a vertex $w\neq u$ of $G-H$ adjacent to $v$. Moreover, if such a vertex $w$ is a leaf, then $(g''(v),g''(u),g''(w))=(2,0,0)$ and $g''(x)=g(x)$ for the other vertices $x$ will be a $\gamma_{cI}(G)$-function, a contradiction to our choice of $g$. Therefore, $w$ is not a leaf. We now let $W$ be the set of all such vertices $w$. Notice that $N(W)\cap V_{0}^{g}$ is an independent set in $H$ in which every vertex is adjacent to precisely one leaf in $G-H$. Therefore, $|W|\geq|N(W)\cap V_{0}^{g}|$ by construction. Now define $h:V(G)\rightarrow\{0,1,2\}$ by $h(v)=2$ for all $v\in V(H)$, and $h(x)=0$ for the other vertices $x$. Clearly, $h$ is a CID function of $G$. Moreover,
$$2\beta(G)=\omega(h)=\omega(g)-|W|-|N(W)\cap V_{0}^{g}|+2|N(W)\cap V_{0}^{g}|\leq \omega(g)=\gamma_{cI}(G).$$
Thus, $\gamma_{cI}(G)=2\beta(G)$.

Conversely, let $\gamma_{cI}(G)=2\beta(G)$ for a graph $G$. Let $S$ be an $\alpha(G)$-set and let $H$ be the subgraph induced by $V(G)\setminus S$. If a vertex $v\in V(H)$ has no leaf in $V(G)\setminus V(H)$, then we efine $g(v)=1$ and $g(x)=f(x)$ for the other vertices $x$. Then $g$ is  a CID function with weight less than $2\beta(G)$, which is a contradiction. Thus at least one leaf in $S$ is adjacent to each vertex in $V(G)\setminus S$. If each vertex of $H$ has at least two leaves in $G-H$, then $G\in \mathcal{G}$, as described above. So, suppose that some vertices $v\in V(H)$ have only one leaf in $S$. Now if $G\not\in\mathcal{G}$, we consider two cases.\vspace{1mm}

\textit{Case $1$.}
$deg_{G\setminus H}(x)=1$, for some $x\in V(H)$. So, $x$ is only adjacent to a leaf $y$ in $S$. If $x$ is an isolated vertex of $H$, then the component $G[x,y]$ contradicts the connectedness of $G$. Therefore, $x$ has at least one neighbor in $H$. Then by assigning $2$ to the vertices in $V(H)\setminus\{x\}$, $1$ to the unique leaf adjacent to $x$ in $G\setminus H$ and $0$ to the other vertices we obtain a CID function with weight $2\beta(G)-1$, which is impossible.\vspace{1mm}

\textit{Case $2$.}
There exist $k$ independent vertices $v_1,v_2,\ldots,v_k$ in $H$, among those ones which have precisely one leaf in $G\setminus H$, with at most $k'<k$ non-leaf independent neighbors $u_1,u_2,\ldots,u_{k'}$ in $G-H$. By assigning $1$ to all vertices $u_i$ and the leaves in $G-H$ adjacent to the vertives $v_j$, $2$ to the vertices in $V(H)\setminus\{v_1,v_2,\ldots,v_k\}$ and $0$ to the other vertices, we have a CID function with weight at most $2\beta(G)-1$. This is also a contradiction. Hence, $G\in \mathcal{G}$.\vspace{1mm}

Consequently, we end up with $G\in \mathcal{G}$, and this completes the proof.
\end{proof}


\section{$K_{1,r}$-free graphs}

In this section we center our attention on presenting some bounds for those graphs containing no stars. A remarkable (and well-known) class of such graphs are the claw-free graphs, to whom we dedicate special attention. We say that a support vertex is \textit{strong} if it is adjacent to more than one leaf.

\begin{theorem}\label{free}
Let $G$ be a $K_{1,r}$-free graph of order $n$ with $s'$ strong support vertices. Then,
$$\gamma_{cI}(G)\geq \frac{2(n+s')}{1+r}$$
and this bound is sharp.
\end{theorem}
\begin{proof}
Let $f$ be a $\gamma_{cI}(G)$-function. We define $A$ as $V_0\cap N(V_{2})$. Since $G$ is $K_{1,r}$-free and $V_0$ is independent, every vertex in $V_{2}$ is adjacent to at most $r-1$ vertices in $A$. This shows that $|A|\leq (r-1)|V_{2}|$. On the other hand, every vertex in $V_0\setminus A$ has at least two neighbors in $V_{1}$. Therefore, $2|V_0\setminus A|\leq|[V_0\setminus A,V_{1}]|\leq(r-1)|V_{1}|$. Moreover, we may assume that all the strong support vertices belong to $V_{2}$. This shows that $|V_{2}|\geq s'$. We now have,
\begin{equation}\label{INE1}
\begin{array}{lcl}
2(n-\gamma_{cI}(G)+s')&\leq&2(n-|V_{1}|-|V_{2}|)\\
&=&2|V_0|=2|A|+2|V_0\setminus A|\leq(r-1)(|V_{1}|+2|V_{2}|)=(r-1)\gamma_{cI}(G).
\end{array}
\end{equation}
So, we deduce $\gamma_{cI}(G)\geq \frac{2(n+s')}{r+1}$.

To see the bound is sharp for $r\geq3$, we consider $G=K_{1,r-1}$ with $\gamma_{cI}(G)=2$, $n=r$ and $s'=1$. Let $r=3$ (in such a case $G$ is claw-free). We begin with a $p$-cycle $v_{1}v_{2}\cdots v_{p}v_{1}$. Add $p$ new vertices $u_{1},\dots,u_{p}$ and join $u_{i}$ to both $v_{i}$ and $v_{i+1}$, for $1\leq i\leq p$, with $v_{p+1}=v_{1}$. It is easy to see that the function $f:V(G)\rightarrow \{0,1,2\}$ defined by $f(v_{i})=1$ for $1\leq i\leq p$, and $f(u_{i})=0$ for $1\leq i\leq p$ is a CID function of $G$ with weight $\gamma_{cI}(G)=n/2=2(n+0)/(1+3)$.
\end{proof}

As an immediate result of Theorem \ref{free} we have the following corollary for all graphs $G$.

\begin{corollary}\label{cor4}
Let $G$ be a graph of order $n$ with $s'$ strong support vertices. Then, $\gamma_{cI}(G)\geq \frac{2(n+s')}{\Delta+2}$ and this bound is sharp.
\end{corollary}


\subsection{Claw-free graphs}

In the rest of this section, we characterize all claw-free graphs for which the equality holds in Theorem \ref{free}. Assume that $G$ is a claw-free graph with components $G_{1},\dots,G_{p}$. Since $\gamma_{cI}(G)=\sum_{i=1}^{p}\gamma_{cI}(G_{i})$, it follows that $\gamma_{cI}(G)=(n+s')/2$ if and only if $\gamma_{cI}(G_{i})=(n_{i}+s_{i}')/2$ for each $1\leq i\leq p$, in which $n_{i}=|V(G_{i})|$ and $s_{i}'$ is the number of strong support vertices of $G_{i}$. Notice that in this case ($r=3$) $s'=\sum_{i=1}^{p}s_{i}'=0$, unless $G_{i}=P_{3}$ for some $1\leq i\leq p$. For such a component $G_{i}$, we have $\gamma_{cI}(G_{i})=(n_{i}+s_{i}')/2$. Hence, in what follows, we may assume that $G$ is connected and $s'=0$.

We recall that a {\em $k$-sun} on $2k$ vertices is a construction starting with a Hamiltonian graph $G$ of order $k$, with Hamiltonian $k$-cycle $v_{1}\cdots v_{k}v_{1}$, then $k$ new vertices $u_{1},\dots,u_{k}$ are added so that each $u_{i}v_{i},u_{i}v_{i+1}\in E(G)$, where $v_{k+1}=v_{1}$. We call a graph of the following form a \textit{$k$-triangle} in which the number of triangles is $k-1$.
\begin{figure}[h]\vspace{-25mm}
\tikzstyle{every node}=[circle, draw, fill=white!, inner sep=0pt,minimum width=.16cm]
\begin{center}
\begin{tikzpicture}[thick,scale=.6]
  \draw(0,0) { 

+(-5.5,-5) node{}
+(-4,-5) node{}
+(-2.5,-5) node{}
+(-1,-5) node{}
+(.5,-5) node{}

+(-4.75,-6) node{}
+(-3.25,-6) node{}
+(-.25,-6) node{}

+(-5.5,-5) -- +(-4,-5) -- +(-2.5,-5)
+(-1,-5) -- +(.5,-5)
+(-4,-5) -- +(-4.75,-6) -- +(-5.5,-5) -- +(-4,-5) -- +(-3.25,-6) -- +(-2.5,-5)
+(-1,-5) -- +(-.25,-6) -- +(.5,-5)

+(-2,-5.2) node[rectangle, draw=white!0, fill=white!100]{${\textbf{.}}$}
+(-1.75,-5.2) node[rectangle, draw=white!0, fill=white!100]{${\textbf{.}}$}
+(-1.5,-5.2) node[rectangle, draw=white!0, fill=white!100]{${\textbf{.}}$}

+(-5.5,-4.6) node[rectangle, draw=white!0, fill=white!100]{$v_{1}$}
+(-4,-4.6) node[rectangle, draw=white!0, fill=white!100]{$v_{2}$}
+(-2.5,-4.6) node[rectangle, draw=white!0, fill=white!100]{$v_{3}$}
+(-1,-4.6) node[rectangle, draw=white!0, fill=white!100]{$v_{k-1}$}
+(.6,-4.6) node[rectangle, draw=white!0, fill=white!100]{$v_{k}$}

};
\end{tikzpicture}
\end{center}\vspace{-5mm}
\caption{A $k$-triangle.}\label{fig1}
\end{figure}
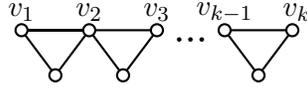

In order to characterize all connected claw-free graphs whose CID numbers equal half of their orders, we introduce $\Omega$ as the family of graphs $G$ satisfying one of the following items.\vspace{1mm}

($i$) $G$ is a $k$-sun in which the Hamiltonian graph is a cycle, or\vspace{.75mm}

($ii$) $G$ is of the form a graph $G(k_{1},\dots,k_{r})$ depicted in Figure \ref{fig2}.

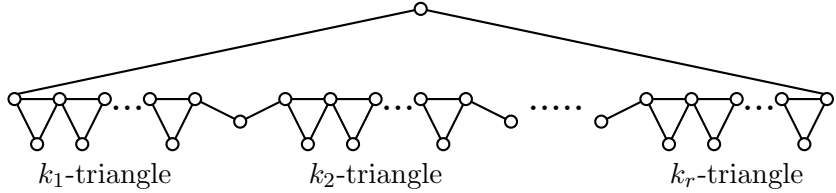
\begin{figure}[h]\vspace{-16mm}
\tikzstyle{every node}=[circle, draw, fill=white!, inner sep=0pt,minimum width=.16cm]
\begin{center}
\begin{tikzpicture}[thick,scale=.6]
  \draw(0,0) { 

+(-19,-5) node{}
+(-18,-5) node{}
+(-17,-5) node{}
+(-16,-5) node{}
+(-15,-5) node{}
+(-14,-5.5) node{}

+(-18.5,-6) node{}
+(-17.5,-6) node{}
+(-15.5,-6) node{}

+(-19,-5) -- +(-18,-5) -- +(-17,-5)
+(-16,-5) -- +(-15,-5) -- +(-14,-5.5)
+(-19,-5) -- +(-18.5,-6) -- +(-18,-5) -- +(-17.5,-6) -- +(-17,-5)
+(-16,-5) -- +(-15.5,-6) -- +(-15,-5)

+(-16.75,-5.2) node[rectangle, draw=white!0, fill=white!100]{ ${\textbf{.}}$}
+(-16.5,-5.2) node[rectangle, draw=white!0, fill=white!100]{ ${\textbf{.}}$}
+(-16.25,-5.2) node[rectangle, draw=white!0, fill=white!100]{ ${\textbf{.}}$}
+(-13,-5) node{}
+(-12,-5) node{}
+(-11,-5) node{}
+(-10,-5) node{}
+(-9,-5) node{}
+(-8,-5.5) node{}

+(-12.5,-6) node{}
+(-11.5,-6) node{}
+(-9.5,-6) node{}

+(-13,-5) -- +(-12,-5) -- +(-11,-5)
+(-10,-5) -- +(-9,-5) -- +(-8,-5.5)
+(-13,-5) -- +(-12.5,-6) -- +(-12,-5) -- +(-11.5,-6) -- +(-11,-5)
+(-10,-5) -- +(-9.5,-6) -- +(-9,-5)
+(-14,-5.5) -- +(-13,-5)

+(-10.75,-5.2) node[rectangle, draw=white!0, fill=white!100]{ ${\textbf{.}}$}
+(-10.5,-5.2) node[rectangle, draw=white!0, fill=white!100]{ ${\textbf{.}}$}
+(-10.25,-5.2) node[rectangle, draw=white!0, fill=white!100]{ ${\textbf{.}}$}
+(-7.5,-5.2) node[rectangle, draw=white!0, fill=white!100]{ ${\textbf{.}}$}
+(-7.25,-5.2) node[rectangle, draw=white!0, fill=white!100]{ ${\textbf{.}}$}
+(-7,-5.2) node[rectangle, draw=white!0, fill=white!100]{ ${\textbf{.}}$}
+(-6.75,-5.2) node[rectangle, draw=white!0, fill=white!100]{ ${\textbf{.}}$}
+(-6.5,-5.2) node[rectangle, draw=white!0, fill=white!100]{ ${\textbf{.}}$}
+(-6,-5.5) node{}
+(-5,-5) node{}
+(-4,-5) node{}
+(-3,-5) node{}
+(-2,-5) node{}
+(-1,-5) node{}
+(-10,-3) node{}

+(-4.5,-6) node{}
+(-3.5,-6) node{}
+(-1.5,-6) node{}

+(-5,-5) -- +(-4,-5) -- +(-3,-5)
+(-2,-5) -- +(-1,-5)
+(-5,-5) -- +(-4.5,-6) -- +(-4,-5) -- +(-3.5,-6) -- +(-3,-5)
+(-2,-5) -- +(-1.5,-6) -- +(-1,-5)
+(-6,-5.5) -- +(-5,-5)
+(-19,-4.9) -- +(-10,-3) -- +(-1,-4.9)

+(-2.75,-5.2) node[rectangle, draw=white!0, fill=white!100]{${\textbf{.}}$}
+(-2.5,-5.2) node[rectangle, draw=white!0, fill=white!100]{${\textbf{.}}$}
+(-2.25,-5.2) node[rectangle, draw=white!0, fill=white!100]{${\textbf{.}}$}

+(-17,-6.7) node[rectangle, draw=white!0, fill=white!100]{$k_{1}$-triangle}
+(-11,-6.7) node[rectangle, draw=white!0, fill=white!100]{$k_{2}$-triangle}
+(-3,-6.7) node[rectangle, draw=white!0, fill=white!100]{$k_{r}$-triangle}

};
\end{tikzpicture}
\end{center}\vspace{-4mm}
\caption{The graph $G(k_{1},\dots,k_{r})$.}\label{fig2}
\end{figure}

We are now in a position to present the main theorem of this section.

\begin{theorem}\label{Claw}
Let $G$ be a connected claw-free graph of order $n\geq4$. Then, $\gamma_{cI}(G)=n/2$ if and only if $G\in \Omega$.
\end{theorem}

\begin{proof}
Suppose first that the equality holds for a connected graph $G$. Then all inequalities in (\ref{INE1}) hold with equality for $(r,s')=(3,0)$, necessarily. In particular, $V_{2}=\emptyset$ by equality instead of the first inequality in (\ref{INE1}). This shows that every vertex in $V_{0}$ has at least two neighbors in $V_{1}$. Taking into account this, the equalities $2|V_{0}|=2|V_{0}\setminus A|=|[V_{0}\setminus A,V_{1}]|$ shows that every vertex in $V_{0}$ has precisely two neighbors in $V_{1}$ (note that $A=\emptyset$ since $V_{2}=\emptyset$). On the other hand, $|[V_{0}\setminus A,V_{1}]|=2|V_{1}|$ follows that every vertex in $V_{1}$ is adjacent to exactly two vertices in $V_{0}$.

Let $H=G[V_{1}]$ be the subgraph of $G$ induced by $V_{1}$. Suppose to the contrary that there exists a vertex $v\in V(H)$ for which $deg_{H}(v)\geq3$. Since $G$ is claw-free and $v$ has precisely two neighbors in $V_{0}$, it follows that every vertex in $N_{H}(v)$ must be adjacent to at least one of those two neighbors of $v$, say $v_{1}$ and $v_{2}$, in $V_{0}$. This implies that $deg(v_{1})\geq3$ or $deg(v_{2})\geq3$, which is a contradiction. The above argument guarantees that $\Delta(H)\leq2$ and therefore $H$ is isomorphic to a disjoint union of some cycles and paths.

Let $H'$ be a $t$-cycle $v_{1}v_{2}\cdots v_{t}v_{1}$ as a component of $H$. Let $v_{11}$ and $v_{12}$ be the neighbors of $v_{1}$ in $V_{0}$. Since $G$ is claw-free and $v_{11}v_{12}\notin E(G)$, both $v_{2}$ and $v_{t}$ have at least one neighbor in $\{v_{11},v_{12}\}$. Moreover, since $deg(v_{11})=deg(v_{12})=2$, both $v_{2}$ and $v_{t}$ have precisely one neighbor in $\{v_{11},v_{12}\}$. We may assume that $v_{11}v_{t},v_{12}v_{2}\in E(G)$. Let $v_{22}$ be the second neighbor of $v_{2}$ in $V_{0}$. Again, since $G$ is claw-free and $v_{12}v_{22},v_{12}v_{3}\notin E(G)$, we have $v_{22}v_{3}\in E(G)$. Iterating this process, it results in a $t$-sun for which $v_{1}v_{2}\cdots v_{t}v_{1}$ is the initiating Hamiltonian cycle. Notice that no vertex of the $t$-sun has other neighbors in $G$. Now $G$ satisfies the condition ($i$) because it is connected.

In what follows, we may assume that no component of $H$ is a cycle. If $H$ is edgeless, then $G$ is isomorphic to the cycle $C_{n}$ in which $n$ is even. It is then easy to see that $G$ is of the form $G(k_{1},\dots,k_{r})$, where $k_{1}=\cdots=k_{r}=1$. Suppose now that $H$ is not edgeless and let $H''$ be a path $v_{1}v_{2}\cdots v_{t}$ as a component of $H$ in which $t\geq2$. Let $v_{21}$ and $v_{22}$ be the neighbors of $v_{2}$ in $V_{0}$. Note that $v_{1}$ must be adjacent to at least one of $v_{21}$ and $v_{22}$, for otherwise we would have a claw as an induced subgraph of $G$. If $v_{1}$ is adjacent to both $v_{21}$ and $v_{22}$, then $G[v_{1},v_{2},v_{21}]$ is a $2$-triangle. In such a case, we have $t=2$ (for otherwise since $G[v_{2},v_{3},v_{21},v_{22}]$ is not an induced claw in $G$, it follows that $v_{22}v_{3}\in E(G)$ or $v_{21}v_{3}\in E(G)$ and therefore $deg(v_{22})\geq3$ or $deg(v_{21})\geq3$ which are impossible). Therefore, $G\cong G[v_{1},v_{2},v_{21},v_{22}]\cong K_{4}-v_{21}v_{22}$. In such a case, $G$ is of the form $G(k_{1},\dots,k_{r})$ in which $k_{1}=2$ and $k_{2}=\cdots k_{r}=0$ since it is connected. So, $G\in \Omega$. In what follows, we assume that $v_{1}$ is adjacent to only one of $v_{21}$ and $v_{22}$. We then continue with $v_{2}$. Since $G$ is claw-free, it follows that both $v_{1}$ and $v_{3}$ must have at least one neighbor in $\{v_{21},v_{22}\}$. Moreover, $deg(v_{21})=deg(v_{22})=2$ shows that both $v_{1}$ and $v_{3}$ have exactly one neighbor in $\{v_{21},v_{22}\}$. So, we may assume that $v_{1}v_{21},v_{3}v_{22}\in E(G)$. Similarly, $v_{3}$ has two neighbors $v_{31}$ and $v_{32}$ in $V_{0}$, in which we may assume that $v_{31}=v_{22}$. By repeating this process we obtain a $t$-triangle on the set of vertices $K=\{v_{1},v_{2},\dots,v_{t},u_{12},u_{23},\dots,u_{(t-1)t}\}$ in which $u_{i(i+1)}$ is adjacent to both $v_{i}$ and $v_{i+1}$, for $1\leq i\leq t-1$. We now distinguish two cases depending on $S=(N(v_{1})\cap N(v_{t}))\setminus\{v_{2}\}$.

\textit{Case $1$.} Suppose that $S\neq\emptyset$. Let $u_{1t}$ be in $S$. Since the path $v_{1}v_{2}\cdots v_{t}$ is a component of $H$, we have $S\cap V(H)=\emptyset$. Therefore, $u_{1t}\in V_{0}$. Since both $v_{1}$ and $v_{t}$ have exactly two neighbors in $V_{0}$, it follows that $S=\{u_{1t}\}$. In such a case, $G$ is isomorphic to the subgraph induced by $K\cup\{u_{1t}\}$ that is of the form $G(t,0,\dots,0)$ because it is connected.

\textit{Case $2$.} Suppose now that $S=\emptyset$. This implies that the subgraph induced by $K$ is a $t$-triangle. Then $v_{t}$ is adjacent to a vertex $x'\neq u_{(t-1)t}$ in $V_{0}$ and $x'$ is adjacent to vertex $x\notin V(H'')$, where $H''$ is the path on the vertices $v_{1},\dots,v_{t}$ (note that if $x\in V(H'')$, then $x=v_{j}\in\{v_{1},\dots,v_{t-1}\}$. If $j=1$, then $G$ is again of the form $G(t,0,\dots,0)$ and hence $G\in \Omega$. If $j\geq2$, then one of the vertices in $\{v_{1},\dots,v_{t-1}\}$ has at least three neighbors in $V_{0}$ which is impossible). Since $V_{0}$ is independent, $x$ belongs to a component $H'''$ of $H$. Since $H$ does not have a cycle as a component, it follows that $H'''$ is a path. In such a case, the vertices of $H'''$ belong to a $|V(H''')|$-triangle by a similar fashion. Iterating this process we obtain some $|V(H''_{1})|,\dots,|V(H''_{r})|$-triangles constructed as above, in which $H''_{1}=H''$, $H''_{2}=H'''$ and $r$ is the largest integer for which there exists such a $|V(H''_{r})|$-triangle. Let $H''_{r}=w_{1}\cdots w_{p}$ and let $w_{p}$ be the vertex which has only one neighbor in $V_{0}$ of the subgraph induced by $V(H''_{1})\cup\cdots\cup V(H''_{r})$. Similar to Case $1$, there exists a vertex $u_{1p}\in V_{0}$ such that both $v_{1}$ and $w_{p}$ are adjacent to it. Therefore, $G$ is of the form $G(|V(H''_{1})|,\dots,|V(H''_{r})|)$. In both cases above we have concluded that $G\in \Omega$.

Conversely, let $G\in \Omega$. Suppose first that $G$ satisfies the condition ($ii$). Let $v_{1},\dots,v_{t}$ be those vertices of degree at lest three of $G(k_{1},\dots,k_{r})$. Then $f(v_{i})=1$ for $1\leq i\leq t$, and $f(v)=0$ for the other vertices defines a CID function with weight half of the order. Therefore, $\gamma_{cI}(G_{1})\leq n/2$. Finally, we suppose that $G$ satisfies the condition ($i$). Let $w_{1}w_{2}\cdots w_{q}w_{1}$ be the $q$-cycle on the vertices of degree four. Then the assignment $g(w_{i})=1$ for $1\leq i\leq q$, and $g(w)=0$ for the other vertices defines a CID function of $G_{2}$ with weigh half of its order. So, $\gamma_{cI}(G_{2})\leq n/2$. Therefore, in both cases we have $\gamma_{cI}(G)=n/2$. This completes the proof.
\end{proof}


\section{Graphs with small or large CID numbers}

Obviously, $2\leq \gamma_{cI}(G)\leq n$ for any connected graph $G$ of order $n\geq2$. Our aim in this section is to characterize the connected graphs $G$ with values for $\gamma_{cI}(G)$ equal or near to the bounds above.

\begin{proposition}\label{pro4}
Let $G$ be a connected graph. Then $\gamma_{cI}(G)=2$ if and only if $G\in \{K_{1,n},\ K_{2,n}\}$ or $G=K^*_{2,n}$, for $n\geq1$, which is obtained from $K_{2,n}$ by joining the two vertices in its $2$-vertex partite set.
\end{proposition}
\begin{proof}
Clearly, $\gamma_{cI}(K_{1,n})=\gamma_{cI}(K_{2,n})=\gamma_{cI}(K^*_{2,n})=2$.

Conversely, let $\gamma_{cI}(G)=2$ and let $f$ be a $\gamma_{cI}(G)$-function. Hence, one of the following situations holds.

(a) There exists a vertex $v$ for which $f(v)=2$. In such a case, the other vertices are assigned $0$ under $f$, and are independent and adjacent to $v$. Therefore $G=K_{1,n}$.

(b) There exist two vertices $v$ and $u$ for which $(f(u),f(v))=(1,1)$ and $f(w)=0$ for the other vertices. Note that the vertices in $V(G)\setminus\{u,v\}$ are independent and adjacent to both $u$ and $v$. Therefore $G\in \{K_{2,n},K^*_{2,n}\}$.
\end{proof}

We now give the characterization of all connected graphs $G$ with $\gamma_{cI}(G)=3$. To this aim, we define the families $\mathcal{R}_{i}$, $1\leq i\leq7$, of graphs $G$ as follows. We first fix some necessary notation for such constructions. For the given vertices $x$, $y$ and $z$ of the graph $G$ in construction, we shall use $V_{x}=\{v\in V(G)\mid N(v)=\{x\}\}$, $V_{x,y}=\{v\in V(G)\mid N(v)=\{x,y\}\}$ and $V_{x,y,z}=\{v\in V(G)\mid N(v)=\{x,y,z\}\}$. Also these sets satisfy that $V_{x}\cup V_{x,y}\cup V_{x,y,z}$ is an independent set.\vspace{1.75mm}

$\mathcal{R}_{1}$: We begin with two nonadjacent vertices $x$ and $y$. Then, to obtain a graph $G\in \mathcal{R}_{1}$ we add two nonempty sets of type $V_{x}$ and $V_{x,y}$.

$\mathcal{R}_{2}$: It is the family of all graphs $G$ obtained by adding the edge $xy$ to the graphs in the family $\mathcal{R}_{1}$.

$\mathcal{R}_{3}$: We begin with three independent vertices $x$, $y$ and $z$. Then a graph $G\in \mathcal{R}_{3}$ satisfies one of the following conditions. ($a_{3}$) $V_{x,y},V_{y,z}\neq \emptyset$, ($b_{3}$) only one of $V_{x,y}$ and $V_{y,z}$ is an empty set and $V_{x,y,z}\neq \emptyset$, ($c_{3}$) $V_{x,y}\cup V_{y,z}=\emptyset$ and $|V_{x,y,z}|\geq3$.

$\mathcal{R}_{4}$: We begin with an edge $xy$ and a vertex $z$. Then, $G\in \mathcal{R}_{4}$ satisfies ($a_{4}$) $V_{y,z}\neq \emptyset$, or ($b_{4}$) $V_{y,z}=\emptyset$ and $V_{x,y,z}\neq \emptyset$.

$\mathcal{R}_{5}$: We begin with a path $xyz$. Then, $G\in \mathcal{R}_{5}$ satisfies ($a_{5}$) $V_{x,y}\cup V_{y,z}\neq \emptyset$, or ($b_{5}$) $V_{x,y}\cup V_{y,z}=\emptyset$ and $|V_{x,y,z}|\geq2$.

$\mathcal{R}_{6}$: We begin with the edges $xy$ and $xz$. Then, one of the following conditions holds for a graph $G\in \mathcal{R}_{6}$: ($a_{6}$) $V_{x,y}\neq \emptyset$, ($b_{6}$) $V_{x,y}=\emptyset$ and both $V_{y,z}$ and $V_{x,y,z}$ are nonempty sets, ($c_{6}$) $V_{x,y}=\emptyset$ and $|V_{x,y,z}|\geq2$.

$\mathcal{R}_{7}$: We begin with a $3$-cycle $xyzx$. Then, $G\in \mathcal{R}_{7}$ satisfies ($a_{7}$) $V_{x,y},V_{y,z}\neq \emptyset$, or ($b_{7}$) $V_{x,y}$ or $V_{y,z}$ equals empty set and $V_{x,y,z}\neq \emptyset$.\vspace{1.75mm}

With these families above in mind, we are able to present the next characterization.

\begin{theorem}\label{theo1}
Let $G$ be a connected graph of order $n$. Then, $\gamma_{cI}(G)=3$ if and only if $G\in \cup_{i=1}^{7}\mathcal{R}_{i}$.
\end{theorem}

\begin{proof}
If $G\in \mathcal{R}_{1}\cup \mathcal{R}_{2}$, then the function $g$ with $g(x)=2$, $g(y)=1$ and $g(v)=0$ for any other distinct vertex is a CID function of $G$ with weight $\gamma_{cI}(G)=3$. If $G\in \cup_{i=3}^{7}\mathcal{R}_{i}$, then the function $g$ with $g(x)=g(z)=g(y)=1$ and $g(v)=0$ for the remaining vertices defines a CID function of $G$ with weight $\gamma_{cI}(G)=3$ too.

Conversely, let $f:V(G)\rightarrow \{0,1,2\}$ be a minimum CID function with weight $\omega(f)=3$. Let $V_i=\{v\in V| f(v)=i\}$ where $i\in \{0,1,2\}$. We consider two cases depending on $V_{2}$.

\textit{Case 1.} $V_2\neq \emptyset$. Hence, there exists a unique vertex $x$ with weight $2$ under $f$ and one vertex $y$ with $f(y)=1$. Note that the other vertices must be assigned $0$ under $f$ and must belong to $V_{x}\cup V_{x,y}$ type sets. If $V_{x}=\emptyset$, then $\gamma_{cI}(G)=2$, which is impossible. Also, if $V_{x,y}=\emptyset$, then either $\gamma_{cI}(G)=2$ or $G$ is disconnected, a contradiction again. Therefore, $G\in \mathcal{R}_{1}\cup \mathcal{R}_{2}$.

\textit{Case 2.} $V_2=\emptyset$. We assume that $(f(x),f(y),f(z))=(1,1,1)$ for some vertices $x$, $y$ and $z$, and $f(v)=0$ for the other vertices $v$. Note that the other vertices $v$ are independent.

Let first $G[\{x,y,z\}]$ be edgeless. If the sets $V_{x,y}$ and $V_{y,z}$ are nonempty, then $G\in \mathcal{R}_{3}$. If both $V_{x,y}$ and $V_{y,z}$ are empty, then $|V_{x,y,z}|\geq3$, otherwise $\gamma_{cI}(G)=2$ or $G$ is disconnected. If only one of $V_{x,y}$ and $V_{y,z}$, say $V_{x,y}$, is empty, then $V_{x,y,z}\neq \emptyset$, for otherwise $G$ would be disconnected. This shows that $G\in \mathcal{R}_{3}$.

Let $y$ be adjacent to only one of $x$ and $z$, say $x$, and $xz\notin E(G)$. If the set $V_{y,z}=\emptyset$, then $V_{x,y,z}\neq \emptyset$, for otherwise $G$ is disconnected. So, $G\in \mathcal{R}_{4}$.

Let $G[\{x,y,z\}]$ be isomorphic to the path $xyz$. Let $V_{x,y}\cup V_{y,z}=\emptyset$. If $|V_{x,y,z}|\leq1$, then $\gamma_{cI}(G)=2$ which is impossible. Thus, $|V_{x,y,z}|\geq2$, and so, $G\in \mathcal{R}_{5}$.

Assume now that $xy,xz\in E(G)$ and $yz\notin E(G)$. Suppose that $V_{x,y}=\emptyset$. Hence, we must have $V_{x,y,z}\neq \emptyset$, for otherwise $\gamma_{cI}(G)=2$. If $|V_{x,y,z}|=1$, then $V_{y,z}\neq \emptyset$, for otherwise we have again $\gamma_{cI}(G)=2$. Therefore, $V_{y,z}\neq \emptyset$. This implies that $G\in \mathcal{R}_{6}$.

Finally, assume that $G[\{x,y,z\}]$ is isomorphic to the $3$-cycle $xyzx$. Suppose that at least one of the sets $V_{x,y}$ and $V_{y,z}$ is empty. Hence, $V_{x,y,z}$ is a nonempty set, for otherwise $\gamma_{cI}(G)=2$. This completes the proof.
\end{proof}

The characterizations of graphs $G$ of order $n$ with $\gamma_{cI}(G)\in \{n-1,n\}$ were given in \cite{fan}. In fact, there was shown that $\gamma_{cI}(G)=n$ if and only if $\Delta(G)\leq1$, and moreover, that $\gamma_{cI}(G)=n-1$ if and only if $G\in \{P_{3},P_{4},K^t_m\}$ where $K^t_m$ is a graph obtained from the complete graph $K_m$ by joining $t$ leaves to $t$ vertices of $K_m$, where $m\geq3$ and $0\leq t\leq m$.

In the rest of this section, we characterize the family of graphs $G$ of order $n$ with $\gamma_{cI}(G)=n-2$. For the sake of convenience, we introduce some necessary notations. Let $H_{0}$ stand for a complete graph. For $m\geq1$, let $H_m$ be a graph obtained from a complete graph by removing $m$ edges subject to $\delta(H_m)\geq2$ and $\alpha(H_m)=2$. For $m\geq2$, let $H_m^k$ be a graph obtained from the graph $H_m$ by joining $k$ leaves to $k$ distinct vertices of $H_m$ in which $0\le k\le n$. Let $H_1^k$ be a graph obtained from the graph $H_1$ by joining $k$ leaves to $k$ or $k-1$ vertices, where $0\le k\le n+1$. Let $H_0^k$ be a graph obtained from a complete graph $K_n$ by joining $k$ leaves to $k-1$ vertices of $K_n$ where $2\le k\le n+1$, and $H_0^1$ be a graph obtained from a complete graph by joining one leaf to one vertex of it.\\
We also consider $P_{4}^k$ as a graph obtained from $P_4$ by joining $k$ leaves to exactly $k$ vertices of $P_4$, where $1\leq k\leq4$.\vspace{2mm}

Let $\mathcal{G}_1$ be the family of graphs $G$ of order $n\geq5$ with $\alpha(G)=2$ for which one of the following conditions holds.

\textbf{1.1.} $G=H_m$ for $m\geq1$, or\vspace{.5mm}

\textbf{1.2.} $G$ has only one leaf $v$ adjacent to a support vertex $u$ of degree $2$ such that $G-\{v,u\}=H_0=K_{n-2}$, or\vspace{.5mm}

\textbf{1.3.} $G$ has only one leaf $v$ adjacent to a support vertex $u$ of degree at least $3$, $G-\{u,v\}=H_0=K_{n-2}$ and the vertex $u$ is not adjacent to all vertices of $H_0$.\vspace{2mm}

Let $\mathcal{G}_2$ be a family of graphs $G$ with $\alpha(G)\geq3$ with at least one leaf, for which one of the following properties holds. In order to be next used, assume $I=\{v_1,v_2,\ldots,v_{\alpha(G)}\}$ is a maximum independent set of $G$.\vspace{.5mm}

\textbf{2.1.} Let all vertices in $I$ be leaves. Moreover, let $G-I=H_0$ and consider $G$ has precisely one support vertex $x$ with two adjacent leaves, and the other support vertices have only one adjacent leaf. In such a situation, we have $G=H_0^{\alpha(G)}$.\vspace{.5mm}

\textbf{2.2.} Let $G-I=H_1=K_r-e$ where $e=uw$. Let every support vertex of $G$ have exactly one adjacent leaf, or one of $u$ or $w$ has two adjacent leaves and the other support vertices have one adjacent leaf.\vspace{.5mm}

\textbf{2.3}. Let $G-I=H_m=K_r-\{uw_1,uw_{2},\cdots,uw_{m}\}$ ($m\geq2$). Let every support vertex of $G$ be adjacent to exactly one leaf, or the vertex $u$ is adjacent to two leaves and the other support vertices have only one adjacent leaf.\vspace{.5mm}

\textbf{2.4}. Let $G-I=H_m=K_r-\{u_1w_1,u_2w_{2},\cdots,u_mw_{m}\}$ for $m\geq2$, in which the set $\{u_1w_1,u_2w_{2},\cdots,u_mw_{m}\}$ has at least two independent edges. Moreover, suppose that any support vertex has exactly one adjacent leaf.\vspace{1mm}

\textbf{3.1.} Suppose that all vertices in $I\setminus \{v_1\}$ are leaves. We set $G''=G-(I\setminus \{v_1\})$. Let $G''=H_{0}$. Moreover, we let $G$ have exactly one support vertex $x$ with two adjacent leaves and the other support vertices have one adjacent leaf. In such a situation, we have $G=H_0^{\alpha(G)-1}$.\vspace{.5mm}

\textbf{3.2.} Suppose that all the vertices in $I\setminus \{v_1\}$ are leaves. Let $G-(I\setminus \{v_1\})=H_m=K_r-\{uw_1,uw_{2},\cdots,uw_{m}\}$ in which $u$ is a support vertex with two adjacent leaves and the other support vertices are adjacent to only one leaf.\vspace{.5mm}

\textbf{3.3.} Suppose that all the vertices in $I\setminus \{v_1\}$ are leaves. We suppose that $G-(I\setminus \{v_1\})=H_m=K_r-\{u_1w_1,u_2w_{2},\cdots,u_mw_{m}\}$ and all support vertices are adjacent to only one leaf.\vspace{.5mm}

\textbf{3.4}. Let exactly two vertices $v_1,v_2\in I$ be not leaves, $G=H_m^{\alpha(G)-2}$ in which $H_m=K_r-\{v_1v_{2}, v_1w_{1},\cdots,v_1w_{m-1}\}$ for $m\geq1$, and every support vertex have exactly one adjacent leaf.

\textbf{3.5}. Let exactly two vertices $v_1,v_2\in I$ be not leaves, $G=H_m^{\alpha(G)-2}$ in which $H_m=K_r-\{v_{1}v_{2},u_{1}w_{1},\cdots,u_{m-1}w_{m-1}\}$ for $m\geq1$, in which $v_{1}v_{2}$ and $u_{k}w_{k}$ have no endpoints in common for some $1\leq k\leq m-1$. Moreover, any support vertex has exactly one leaf as a neighbor.

\begin{theorem}\label{pro7}
Let $G$ be a connected graph of order $n$. Then $\gamma_{cI}(G)=n-2$ if and only if $G\in\{K_{1,3},C_{4},K_{4}-e,P_4^k\}\cup \mathcal{G}_1\cup \mathcal{G}_2$, where $e$ is an arbitrary edge.
\end{theorem}

\begin{proof}
Let $G\in\{K_{1,3},C_{4},K_{4}-e,P_4^k\}\cup \mathcal{G}_1\cup \mathcal{G}_2$. Hence, by using some calculations and observations we deduce that $\gamma_{cI}(G)=n-2$.

Conversely, let $\gamma_{cI}(G)=n-2$ and $G\notin \{K_{1,3},C_{4},K_{4}-e,P_4^k\}$. This shows that $n\geq5$. Also we have $\alpha(G)\geq 2$, for otherwise $\gamma_{cI}(G)\geq \beta(G)=n-\alpha(G)\geq n-1$ which is impossible. We consider two cases depending on the value $\alpha(G)$.\vspace{1.75mm}

\textbf{Case 1.} $\alpha(G)=2$. We shall show that $G\in \mathcal{G}_1$. We note that $G$ has at most one leaf. Suppose first that $G$ has a leaf $v$ with support vertex $u$. If $G=H_0^1$ where $H_0=K_{n-1}$, then the assignment $0$ to $u$ and $1$ to other vertices gives a $\gamma_{cI}(G)$-function with weight $n-1$. So, this situation cannot happen. If $deg(u)=2$, then $u$ must be adjacent to only one vertex of $H_{0}=K_{n-2}$ since $\alpha(G)=2$. Therefore, $G$ satisfies the condition \textbf{1.2}. We now assume that $deg(u)\geq3$. Since $\alpha(G)=2$, all neighbors of $u$ different from $v$ belong to $H_{0}=K_{n-2}$. On the other hand, if $u$ is adjacent to all vertices of $H_{0}$, then $\gamma_{cI}(G)=n-1$. This is a contradiction. Therefore, $G$ satisfies the condition \textbf{1.3}.

Suppose now that $G$ has no leaves. Then, $G\cong H_{m}$ for some $m\geq1$. Thus, in all three situations above we have $G\in \mathcal{G}_1$.\vspace{1.75mm}

\textbf{Case 2.} $\alpha(G)\geq3$. Let $I=\{v_1,v_2,\ldots,v_{\alpha(G)}\}$ be a maximum independent set of $G$. Then, $G$ has at least one leaf, for otherwise Observation \ref{obser} results in $\gamma_{cI}(G)=n-\alpha(G)\le n-3$. Moreover, we may assume that $I$ contains all leaves of $G$. If $I$ has at least three non-leaf vertices, then by assigning $0$ to them and $1$ to the other vertices we get $\gamma_{cI}(G)\leq n-3$. This is a contradiction. Therefore, at most two vertices in $I$ are not leaves. Now depending on the number of non-leaf vertices in $I$ there are three subcases to be met.\vspace{1mm}

\textbf{Subcase 2.1.} All the vertices in $I$ are leaves. In this situation, each vertex of $G$ is either a support vertex or a leaf. We set $G'=G-I$. It is easy to see that $\alpha(G')\leq2$, for otherwise $\gamma_{cI}(G)\leq n-3$. We now distinguish two possibilities depending on $\alpha(G')$.

\textbf{Subcase 2.1.1.} Let $G'=H_0$ (or, equivalently $\alpha(G')=1$). Suppose that every support vertex $w$ has only one leaf. Then the assignment $0$ to one vertex of $G'$, and $1$ to the remaining vertices leads to a $\gamma_{cI}(G)$-function with weight $n-1$, a contradiction.

Let a support vertex $u$ have more than two leaves. Hence, the assignment value $2$ to $u$, $0$ to its leaves and a vertex of $H_0$ different from $u$, and $1$ to the other vertices gives a CID function with weight $n-3$. This is a contradiction. A similar argument shows that if at least two support vertices are adjacent to two leaves, then $\gamma_{cI}(G)\leq n-3$ as well.

Therefore, one support vertex has precisely two leaves and the other support vertices are adjacent to only one leaf.\vspace{1mm}

\textbf{Subcase 2.1.2.} Let $\alpha(G')=2$. We consider the following possibilities.

\textbf{Subcase 2.1.2.1.} Let $G'=H_1=K_r-e$, where $e=uw$. Suppose to the contrary that the condition \textbf{2.2} does not hold. So, we shall deal with one of the following situations.

($a$) One of the support vertices has at least three leaves or at least two support vertices have more than one leaf. This implies that $\gamma_{cI}(G)\le n-3$, a contradiction.

($b$) A support vertex $x\neq u,w$ is adjacent to two leaves. Then, by assigning $2$ to $x$, $0$ to its adjacent leaves and both $u$ and $w$, and $1$ to the other vertices, it results in $\gamma_{cI}(G)\leq n-3$. This is again a contradiction.\vspace{1mm}

\textbf{Subcase 2.1.2.2.} Let $G'=H_m=K_r-\{u_{1}w_1,u_{2}w_{2},\dots,u_{m}w_{m}\}$ for $m\geq2$. We consider the following situations.

($c$) Assume that $u_{1}=\cdots=u_{m}=u$. Suppose to the contrary that the condition \textbf{2.3} does not hold. If the vertex $u$ has at least three leaves, then the assignment $2$ to $u$, $0$ to its leaves and $w_1$, and $1$ to the other vertices gives us a $\gamma_{cI}(G)$-function with weight at most $n-3$. This is a contradiction. Therefore, $u$ is adjacent to at most two leaves. Now, let $x\neq u$ be adjacent to more than one leaf. Moreover, we may assume that $x\neq w_{2},\dots,w_{m}$. Hence, by assigning $2$ to $x$, $0$ to its leaves and to $u$ and $w_{2}$, and $1$ to the other vertices, we get a CID function with weight $n-3$. This is impossible as well.

($d$) We now suppose that the edges $u_{1}w_1,u_{2}w_{2},\dots,u_{m}w_{m}$ are not incident with one vertex. Let $x$ be a support vertex that has more than one leaf. Since there exists at least one edge $u_{k}w_{k}$ in $\{u_1w_1,u_2w_{2},\dots,u_mw_{m}\}$ which is not incident with $x$, the assignment $2$ to $x$, $0$ to its leaves and to both $u_k$ and $w_k$, and $1$ to the other vertices leads to a $\gamma_{cI}(G)$-function with weight at most $n-3$. This is a contradiction. Thus, the condition \textbf{2.4} is true.\vspace{1mm}

From now on, we assume that at least one of the vertices in $I$ is not a leaf. In such a case, we have two possible subcases.

\textbf{Subcase 2.2.} Assume all vertices in $I$ except $v_{1}$ are leaves. Let $G''=G-(I\setminus\{v_1\})$. Similarly to Subcase $2.1$, we have $\alpha(G'')\leq2$. Also, we have two possibilities depending on $\alpha(G'')\in \{1,2\}$.

\textbf{Subcase 2.2.1.} Let $G''=H_0$. Suppose to the contrary that the condition \textbf{3.1} does not hold. Assume that every support vertex $w$ has only one adjacent leaf. Hence, assigning $0$ to $v_1$ and $1$ to other vertices produces a $\gamma_{cI}(G)$-function with weight $n-1$, a contradiction.

Suppose now that there exist at least two support vertices in $G$ having more than one leaf. By assigning $2$ to these support vertices, $0$ to their leaves and to $v_1$, and $1$ to the other vertices we get a CID function with weight $n-3$, which is impossible.

Assume now there exists one support vertex $x$ in $G$ with at least three leaves. Then, by assigning $2$ to $x$, $0$ to its leaves and to $v_1$, and $1$ to the other vertices we obtain a CID function with weight $n-3$, which is again impossible. Thus, the condition \textbf{3.1} is true.

\textbf{Subcase 2.2.2.} Consider $\alpha(G'')=2$. Suppose then, that $S$ and $L$ are the sets of support vertices and leaves, respectively. If there exist two non-adjacent vertices $x$ and $y$ in $V(G)\setminus (S\cup L)$, then $L\cup\{x,y\}$ is an independent set with $|L\cup\{x,y\}|>\alpha(G)$, which is impossible. Therefore, the subgraph induced by $V(G)\setminus (S\cup L)$ is a complete graph. This shows that $G''=H_{m}=K_{r}-\{u_{1}w_{1},\dots,u_{m}w_{m}\}$, in which at least one of the endpoints of $u_{k}w_{k}$ is a support vertex, for each $1\leq k\leq m$.

Similarly to the arguments given in Subcase $2.2.1$, at most one support vertex has two adjacent leaves and the other ones have precisely one adjacent leaf. Now let $u$ be a support vertex adjacent to two leaves. Suppose that there are two vertices $x$ and $y$ in $V(G'')\setminus\{u\}$ such that $xy\notin E(G'')$. Then, by assigning $2$ to $u$, $0$ to its leaves and to $x$ and $y$, and $1$ to the other vertices, we get a CID function with weight $n-3$, a contradiction. Therefore, the subgraph induced by $V(G'')\setminus\{u\}$ is a complete graph. This shows that $G-(I\setminus \{v_1\})=H_m=K_r-\{uw_1,uw_{2},\cdots,uw_{m}\}$, and so, the condition \textbf{3.2} is valid.

Assume now that all support vertices are adjacent to only one leaf. Therefore, we meet the condition \textbf{3.3}.\vspace{1mm}

\textbf{Subcase 2.3.} Consider now that all vertices in $I$ except $v_{1}$ and $v_{2}$ are leaves. We consider $G'''=G-(I\setminus\{v_1,v_2\})$. Note that if there exists a support vertex $x$ having more than one adjacent leaf, then by assigning $2$ to $x$, $0$ to its leaves and to $v_{1}$ and $v_{2}$, and $1$ to the other vertices, we create a CID function with weight $n-3$. This is a contradiction. So, all support vertices are adjacent to only one leaf. There are now the following possibilities depending on the removed edges from $G'''$ with $\alpha(G''')=2$.

($e$) $G'''=K_r-\{v_1v_{2},v_1w_{1},\dots,v_1w_{m-1}\}$ for $m\geq1$, or

($f$) $G'''=K_r-\{v_{1}v_{2},u_{1}w_{1},\dots,u_{m-1}w_{m-1}\}$ in which $v_{1}v_{2}$ and $u_{k}w_{k}$ have no endpoints in common for some $1\leq k\leq m-1$.

All in all, we have proved that $G\in \mathcal{G }_2$ when $\gamma_{cI}(G)= n-2$ and $\alpha(G)\ge 3$. This completes the proof.
\end{proof}


\section{Conclusions and problems}

$\bullet$ For any graph $G$, we have $\beta(G)\leq \gamma_{cI}(G)\leq2\beta(G)$ as already noted in Remark \ref{vertex-cover}. The family of all graphs $G$ for which $\gamma_{cI}(G)=2\beta(G)$ was characterized in this paper. But the problem of characterizing all graphs for which the equality holds in the lower bound is still open.\vspace{1.5mm}

\textbf{Remark.} A {\em $2$-outer-independent dominating set} ($2$OID set) in a graph $G$ is a subset $S\subseteq V(G)$ such that every vertex in $V(G)\setminus S$ has at least two neighbors in $S$, and the set $V(G)\setminus S$ is independent. The {\em $2$-outer-independent domination number} ($2$OID number) $\gamma_{2}^{oi}(G)$, of a graph $G$, is the smallest possible cardinality among all $2$OID sets in $G$. This concept was introduced in \cite{jk} and investigated in \cite{mpsy}. Clearly, $\gamma_{cI}(G)\leq \gamma_{2}^{oi}(G)$. Now, let $\gamma_{cI}(G)=\beta(G)$ and let $f$ be a $\gamma_{cI}(G)$-function. Since $\gamma_{cI}(G)=|V_{2}|+|V_{2}|+|V_{1}|\geq|V_{2}|+\beta(G)$, we have $V_{2}=\emptyset$. This implies that $\gamma_{2}^{oi}(G)\leq \gamma_{cI}(G)$. Thus, $\gamma_{cI}(G)=\gamma_{2}^{oi}(G)$ when $\gamma_{cI}(G)=\beta(G)$. Therefore, a problem equivalent to the above-mentioned one is to characterize the graphs $G$ for which $\gamma_{2}^{oi}(G)=\beta(G)$.\vspace{2mm}

$\bullet$ Characterize all $K_{1,r}$-free graphs ($r\geq3$) for which the lower bound given in Theorem \ref{free} holds with equality. For the case $r=3$ (claw-free graphs), this problem was solved by Theorem \ref{Claw}.\vspace{2mm}

$\bullet$ It was mentioned that $\gamma_{cI}(G)=\beta(G)$ when $\delta(G)\geq2$. On the other hand, the problem of computing $\beta(G)$ for a bipartite graph $G$ can be solved in polynomial-time (see \cite{gj}). So, $\gamma_{cI}(G)$ can be computed in polynomial-time for this family of graphs, as well. Now the question is: does the CID problem for bipartite graphs with minimum degree one belong to NP? Also, is it possible to construct a polynomial algorithm for computing the value of $\gamma_{cI}(T)$ for any tree $T$?\vspace{2mm}

$\bullet$ An {\em outer independent Roman dominating function} (OIRD function) of a graph $G$ is a Roman dominating function for which $V_{0}$ is independent. The {\em outer independent Roman domination number} (OIRD number) $\gamma_{oiR}(G)$ is the minimum weight among all OIRD functions of $G$. This parameter was introduced in \cite{acs1}. Clearly, $\gamma_{cI}(G)\leq \gamma_{oiR}(G)$ for any graph $G$. A natural problem is to characterize the graphs (or at least the trees) for which these two parameters are equal.



\begin{thebibliography}{99}

\bibitem {acs1} H.A. Ahangar, M. Chellali and V. Samodivkin, {\em Outer independent Roman dominating functions in graphs}, Int. J. Comput. Math. {\bf 94} (2017), 2547--2557.
\bibitem {chhm} M. Chellali, T.W. Haynes, S.T. Hedetniemi and A.A. McRaee, {\em Roman $\{2\}$-domination}, Discrete Appl. Math. {\bf 204} (2016), 22--28.
\bibitem {cdhh} E.J. Cockayne, P.A. Dreyer, S.M. Hedetniemi and S.T. Hedetniemi, {\em Roman domination in graphs}, Discrete Math. {\bf 278} (2004), 11--22.
\bibitem{erdos} P. Erd\"{o}s and T. Gallai, \emph{On the minimal number of vertices representing the edges of a graph}, Publ. Math. Inst. Hungar. Acad. Sci.  \textbf{6} (1961), 181--202.
\bibitem {fan} W. Fan, A. Ye, F. Miao, Z. Shao, V. Samodivkin and S.M. Sheikholeslami, {\em Outer-Independent Italian Domination in Graphs}, IEEE Access, {\bf 7} (2019), 22756--22762.
\bibitem {gj} M.R. Garey and D.S. Johnson, Computers and intractability: A guide to the theory of NP-completeness, W.H. Freeman $\&$ Co., New York, USA, 1979.
\bibitem{garey} M.R. Garey, D.S. Johnson, and L. Stockmeyer, Some simplified NP-complete problems, Proceedings of the Sixth Annual ACM Symposium on Theory of Computing.  (1974), pp. 47–63. doi:10.1145/800119.803884.
\bibitem {hhs2} T.W. Haynes, S.T. Hedetniemi and P.J. Slater, Fundamentals of Domination in Graphs,  Marcel Dekker, New York, 1998.
\bibitem {henning} M.A. Henning and W.F. Klostermeyer, {\em Italian domination in trees}, Discrete Appl. Math. {\bf 217} (2017), 557--564.
\bibitem {jk} N. Jafari Rad and M. Krzywkowski, {\em $2$-outer-independent domination in graphs}, Natl. Acad. Sci. Lett. {\bf 38} (2015), 263--269.
\bibitem {kl} W.F. Klostermeyer and G. MacGillivray, {\em Roman, Italian, and $2$-Domination}, J. Combin. Math. Combin. Comput. {\bf 108} (2019), 125--146.
\bibitem {mpsy} D.A. Mojdeh, I. Peterin, B. Samadi and I.G. Yero, {\em On three outer-independent domination related parameters in graphs}, submitted.
\bibitem {s} I. Stewart, {\em Defend the Roman Empire!}, Sci. Amer. {\bf 281} (1999), 136--139.
\bibitem {we} D.B. West, Introduction to Graph Theory (Second Edition), Prentice Hall, USA, 2001.

\end{thebibliography}
\end{document}